\pdfoutput=1
\documentclass{article}
\usepackage{graphicx}
\usepackage{amsmath,amsfonts,amsthm}
\usepackage{tabu}
\usepackage{booktabs}
\usepackage{makecell}

\usepackage{marvosym}
\usepackage{mathptmx}
\newcommand{\E}{\mathbb{E}}
\newcommand{\R}{\mathbb{R}}
\newcommand{\F}{\mathcal{F}}
\newcommand{\Prob}{\mathbb{P}}
\newcommand{\pa}{\partial}
\newcommand{\tr}{\mathbf{t}}
\newcommand{\Tr}{\mathbf{T}}
\DeclareMathOperator*{\argmin}{arg\,min}

\newtheorem{theorem}{Theorem}[section]

\theoremstyle{definition}
\newtheorem{definition}[theorem]{Definition}
\theoremstyle{remark}
\newtheorem{remark}[theorem]{Remark}

\numberwithin{equation}{section}

\newcommand{\tone}{
	\includegraphics{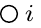}
}

\newcommand{\ttwo}{
	\includegraphics{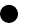}
}

\newcommand{\tthree}{
	\includegraphics{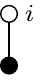}
}

\newcommand{\tfour}{
	\includegraphics{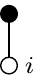}
}

\newcommand{\tfive}{
	\includegraphics{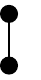}
}

\newcommand{\tsix}{
	\includegraphics{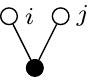}
}

\newcommand{\tseven}{
	\includegraphics{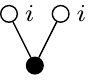}
}

\newcommand{\teight}{
	\includegraphics{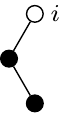}
}

\newcommand{\tnine}{
	\includegraphics{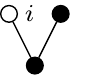}
}

\newcommand{\tten}{
	\includegraphics{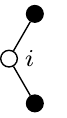}
}

\title{Stochastic symplectic Runge--Kutta methods for the strong approximation of Hamiltonian systems with additive noise}

\author{Weien Zhou\thanks{College of Science, National University of Defense Technology, Changsha 410073, China. (weienzhou@outlook.com).}, Jingjing Zhang\thanks{School of Mathematics and Information Science, Henan Polytechnic University, Jiaozuo 454001, China. (zhangjj@hpu.edu.cn).}, Jialin Hong\thanks{Institute of Computational Mathematics and Scientific/Engineering Computing, Chinese Academy of Sciences, Beijing, 100190, China. (hjl@lsec.cc.ac.cn).}, Songhe Song\thanks{State Key Laboratory of High Performance Computing, National University of Defense Technology, Changsha 410073, China. (shsong@nudt.edu.cn).}}
\date{}

\begin{document}
\maketitle

\begin{abstract}
In this paper, we construct stochastic symplectic Runge--Kutta (SSRK) methods of high strong order for Hamiltonian systems with additive noise. By means of colored rooted tree theory, we combine conditions of mean-square order 1.5 and symplectic conditions to get totally derivative-free  schemes. We also achieve mean-square order 2.0 symplectic schemes for a class of second-order Hamiltonian systems with additive noise by similar analysis. Finally, linear and non-linear systems are solved numerically, which verifies the theoretical analysis on convergence order. Especially for the stochastic harmonic oscillator with additive noise, the linear growth property can be preserved exactly over long-time simulation.\\

\bf{Keywords: Stochastic differential equation; Stochastic Runge--Kutta method; Symplectic integrator; Mean-square convergence}

\end{abstract}

\section{Introduction}
Stochastic differential equations (SDEs) have wide applications in many disciplines like physics, engineering, finance, etc.,
when we take stochastic perturbation into consideration. However,
it is difficult to find explicit solutions of SDEs analytically. There has been tremendous interest in developing effective and reliable numerical methods for SDEs during the last few decades (see e.g. monographs \cite{Kloeden1992,Milstein2004}).
Runge--Kutta (RK) methods are an important family of one-step numerical integrators for ordinary differential equations (ODEs), and recently they have been extended to solve SDEs both for strong approximations \cite{Burrage2000,Wangpeng2008three,Roßler2010,Komori2013rock} and weak approximations \cite{Roßler2009,Debrabant2010,Komori2013,Tang2016}. Especially, the colored rooted tree theory can be applied \cite{Burrage2000,Roßler2010}, which provides an intuitive way to study the order conditions of stochastic Runge--Kutta (SRK) methods.

There exist a variety of crucial issues in designing practical and reliable numerical schemes for SDEs, especially the preservation of dynamics or geometric structures over long time. A notable special case which will be of interest here is the case where the SDEs are even dimensional $(d=2n)$, such that the SDEs possess the following form: 
\begin{equation}\label{e:chamilton}
	\left\{
	\begin{aligned}
		dX(t) &= J^{-1}\nabla H_0 \big(t,X(t)\big)dt + \sum_{r=1}^{m} J^{-1}\nabla H_r \big(t,X(t)\big)\circ dW_r(t),\\
		X(0) &= X_0,
	\end{aligned}\right.
\end{equation}
where  $J=\bigl(\begin{smallmatrix}0& I_n\\-I_n&0
\end{smallmatrix}\bigr)$ is the standard $2n$-dimensional symplectic matrix with $n$-dimensional  identity matrix $I_n$,  $H_i(t,p,q)$ $(i=0,\dots, m)$ are differentiable functions, and $W_r(t) (r=1,\dots,m)$ are standard independent Wiener processes defined on some probability space $(\Omega, \F, \Prob)$. If we denote by $X(t) = \big(P(t), Q(t)\big)^\top$ the solution of \eqref{e:chamilton}, then it can be rewritten as
\begin{equation}\label{e:hamilton}
	\left\{
	\begin{aligned}
		dP(t) &= -\frac{\pa H_0}{\pa q} dt - \sum_{r=1}^{m} \frac{\pa H_r}{\pa q}\circ dW_r(t), \quad & P(0)=p_0,\\
		dQ(t) &= \phantom{-}\frac{\pa H_0}{\pa p} dt + \sum_{r=1}^{m} \frac{\pa H_r}{\pa p}\circ dW_r(t), \quad & Q(0)=q_0.
	\end{aligned}\right.
\end{equation}
This type of SDEs is called a stochastic Hamiltonian system, whose solution is a phase flow almost surely \cite{Milstein2004}. Introduce the differential 2-form
\begin{equation}
	\omega^2 = dp\wedge dq = \sum_{i=1}^{d} dp^i\wedge dq^i,
\end{equation}
and it turns out that the phase flow  of system \eqref{e:hamilton} preserves the symplectic structure (see \cite[Chapter 4]{Milstein2004} for details)
\begin{equation}
	dP(t)\wedge dQ(t) = dp\wedge dq,
\end{equation}
which is an extension of the remarkable property of deterministic Hamiltonian systems \cite{Hairer2006}.
Thus, it is natural to construct numerical integrators inheriting this symplectic property as well. From this point of view, a numerical method with approximation $\{P_n,Q_n\}$ is symplectic provided
\begin{equation}
	dP_{n+1}\wedge dQ_{n+1} = dP_{n}\wedge dQ_{n}.
\end{equation}

Hamiltonian systems perturbed by external Gaussian noises, especially the second-order systems due to Newton's second law of motion  \cite{Milstein2004,Burrage2007,Gitterman2005}, are common and significant in scientific applications. This type of systems help to describe the traditional Hamiltonian systems driven by random forces, which may give rise to essential differences in dynamical evolutions (especially stochastic oscillators \cite{Cruz2016}). Many efforts have been made to construct numerical methods focusing on this type of systems.
A series of symplectic methods is obtained by adding stochastic terms in the deterministic symplectic RK type methods in \cite{Milstein2002add}. \cite{Burrage2012} extends the ideas of Hamiltonian boundary value methods to construct low rank symplectic RK methods. Considering preserving the expectation of the Hamiltonian, \cite{Burrage2014} proposes a class of effective SRK methods.
In \cite{Ma2012}, the authors also consider a class of SRK methods in low stage case for stochastic Hamiltonian systems using stochastic Taylor expansion, which are of mean-square order 1.0. For stochastic oscillators with high frequency, \cite{Cohen2012} proposes an approach based on the variation-of-constants formula, which permits the use of large step-size. In \cite{Sun2017}, the authors propose a kind of numerical methods based on the Pad\'{e} approximations for two kinds of linear stochastic Hamiltonian systems.

Based on \cite{Burrage2000, Roßler2010}, we propose a class of stochastic symplectic Runge--Kutta methods, totally derivative-free, for Hamiltonian systems with additive noise, which are able to reach mean-square order 1.5 and 2.0 in some special cases. The key point to achieve high mean-square order convergence is the additional increments embedded in our schemes. To analyze the convergence order, the main technique we use here is the colored rooted tree theory in the sense of It\^o, so that the order conditions are quite intuitive and flexible to derive.

This paper is organized as follows. In Section \ref{sec:rooted_tree}, the colored rooted tree theories for SDEs are briefly reviewed, which can be used to construct our high mean-square order methods later. Section  \ref{sec:order_add} gives  order conditions for SRK methods aiming at SDEs with additive noise, under which some classes of mean-square order 1.5 schemes are proposed. For Hamiltonian systems with additive noise, the symplectic conditions are given in Section \ref{sec:sym_condition}, combined with which we present the SSRK methods. In Section \ref{sec:2order_add}, we pay attention to a special form of second-order Hamiltonian systems with additive noise. According to its elegant structure, we simplify the order conditions and obtain mean-square order 2.0 without further effort. Finally, in Section \ref{sec:num}, numerical experiments are performed for linear and non-linear systems in order to verify the foregoing order conditions and geometric properties especially the linear growth property.

\section{The colored rooted tree theory}\label{sec:rooted_tree}
Without loss of generality, we restrict consideration to autonomous systems (the coefficients of SDEs do not depend on $t$ explicitly) in this part.
First, we recall some basic facts about SRK methods for general $d$-dimensional SDE in the sense of It\^o:
\begin{equation}\label{e:sdeIto}
\left\{
\begin{aligned}
dX(t) &= f\big(X(t)\big)dt + \sum_{r=1}^m g_r\big(X(t)\big) dW_r(t), \quad t\in [0, T],\\
X(0) &= X_0,
\end{aligned}\right.
\end{equation}
where $f$ and $g_r$, $r=1,\dots,m$, are $\R^d$-valued functions fulfilling a global Lipschitz condition.

The basic tool of constructing our numerical methods is the colored rooted tree theory in \cite{Roßler2010} which is an extension of \cite{Burrage2000} to analyze the order conditions of SRK methods. Thus, we briefly list some definitions and theorems used in constructing numerical schemes later.

Let $\tau_0 = \ttwo$ denote the deterministic node and $\tau_i = \tone$  denote the stochastic node with color $i$, where the subscript $i$ ($i=1,\dots,m$) is actually associated with the $i$th component of the driving Wiener process of SDE \eqref{e:sdeIto}. Let $\Tr$ be the set of all rooted trees with $m+1$ colors $(0,1,\dots,m)$ and $\phi\in \Tr$ stand for the empty tree (the tree without any node).  Moreover, Let $\mathbf{t} = [\mathbf{t}_1,\dots,\mathbf{t}_l]_i \in \Tr$ ($i=0,\dots,m$) be the tree obtained by grafting the roots of subtrees $\mathbf{t}_1,\dots,\tr_l\in \Tr$ each to a common root with node $\tau_i$. For a tree $\mathbf{t}\in \Tr$, let $d(\mathbf{t})$, $s(\mathbf{t})$ denote the number of $\tau_0$ and $\tau_i$ ($i\neq 0$) respectively, so that the order $\rho(\mathbf{t})$ of a tree $\mathbf{t}$ is defined by $\rho(\mathbf{t}) = d(\mathbf{t}) + \frac{1}{2}s(\mathbf{t})$ and $\rho(\phi)=0$. Specific examples of this kind of colored rooted trees will be given in Section \ref{sec:order_add}.

\begin{definition}
	For each tree $\tr\in \Tr$, the elementary differential, a vector-valued function $F(\tr): \R^d \rightarrow \R^d$, is defined recursively as follows.
	\begin{enumerate}
		\item $F(\phi)(x) := x$, i.e., $F(\phi)$ is the identity mapping.
		\item $F(\tau_0)(x) := f(x)$, $F(\tau_i)(x) := g_i(x)$ $(i=1,\dots,m)$ for a single node.
		\item  For a tree $\tr\in \Tr$ with more than one node,
		\begin{equation}
		F(\tr)(x) := \left\{
		\begin{aligned}
		f^{(l)}(x)\Big(F(\tr_1)(x), \dots, F(\tr_l)(x)\Big) \quad &\text{for} \quad \tr=[\tr_1,\dots,\tr_l]_0,\\
		g_i^{(l)}(x)\Big(F(\tr_1)(x), \dots, F(\tr_l)(x)\Big) \quad &\text{for}\quad \tr=[\tr_1,\dots,\tr_l]_i,
		\end{aligned}\right.
		\end{equation}
	\end{enumerate}
	where $f^{(l)}$ and $g_i^{(l)}$ are the symmetric $l$-linear differential operators. For instance, the $\nu$-th element of $f^{(l)}(F(\tr_1),\dots,F(\tr_l)\big)$ is 
	\[\Big(f^{(l)}(F(\tr_1),\dots,F(\tr_l)\big)\Big)_\nu = \sum_{J_1,\dots,J_l=1}^{d}\frac{\pa^l f_\nu}{\pa x^{J_1}\cdots\pa x^{J_l}} \big( F(\tr_1)_{J_1},\dots,F(\tr_l)_{J_l} \big).\]
\end{definition}

\begin{definition}
	Recursively define a multiple stochastic integral $I(\tr)(h)$ for each tree $\tr\in \Tr$ and $h>0$  as 
	\begin{equation}
	I(\tr)(h) :=\left\{
	\begin{aligned}
	& 1, &\text{for} \quad \tr &=\phi,\\
	& \int_{0}^{h}\prod_{j=1}^{l}I(\tr_j)(s)\,ds, &\text{for} \quad \tr&=[\tr_1,\dots,\tr_l]_0, \\
	& \int_{0}^{h}\prod_{j=1}^{l}I(\tr_j)(s) \,dW_i(s), &\text{for} \quad \tr&=[\tr_1,\dots,\tr_l]_i.
	\end{aligned}\right.
	\end{equation}
Moreover, let $I_{i_1,\dots,i_k}(h) := I([\dots[i_1]\dots]_{i_k})(h) $ represent the It\^o multiple integral where integration with respect to $dW_0(s):=ds$ if $i_j=0$, or $dW_q(s)$ if $i_j=q$ $ (j=1,\dots,k,q=1,\dots,m)$:
	\begin{equation}
	I_{i_1,\dots,i_k} (h) = \int_{0}^{h}\int_{0}^{s_{k-1}}\cdots\int_{0}^{s_1} dW_{i_1}(s)\cdots dW_{i_{k-1}}(s_{k-2})\,dW_{i_{k}}(s_{k-1}).
	\end{equation}
\end{definition}

Then the solution $X(h)$ of system \eqref{e:sdeIto} can be represented by a B-series of $F(\tr)(X_0)$ and $I(\tr)(h)$ for $\tr\in \Tr$ \cite{Roßler2010}. Next we can also give expansion of the following form of $s$-stage SRK methods for system \eqref{e:sdeIto} \cite{Burrage2000}
\begin{equation}\label{e:srk}
\left\{
\begin{aligned}
Y_i &= y_n + h\sum_{j=1}^{s}a_{ij}f(Y_j) + \sum_{k=1}^{m}\sum_{j=1}^{s}Z_{ij}^{(k)}g_k(Y_j), \quad i=1,\dots,s,\\
y_{n+1} &= y_n + h\sum_{i=1}^{s}\alpha_{i}f(Y_i) +\sum_{k=1}^{m}\sum_{i=1}^{s}z_i^{(k)}g_k(Y_i),
\end{aligned}\right.
\end{equation}
which can be simply represented in tableau form as
\begin{equation}
\begin{tabular}{l|llll}
& $Z^{(0)}$      & $Z^{(1)}$ & \dots   & $Z^{(m)}$ \\ \hline
& $z^{(0)}$ & $z^{(1)}$ & $\dots$ & $z^{(m)}$
\end{tabular}
\end{equation}
where $Z^{(0)}=hA$ and $z^{(0)}=h\alpha$ denote a matrix and an update vector of deterministic coefficients while matrices $Z^{(1)},\dots,Z^{(m)}$ and update vectors $z^{(1)},\dots,z^{(m)}$ have elements which are certain random variables.

\begin{definition}
	For every $\tr =  [\tr_1,\dots,\tr_u]_k\in \Tr$,  $k\in\{0,\dots, m\}$, let 
	\begin{equation}
	\Phi(\tr) := z^{(k)^\top} \prod_{i=1}^{u}\Psi(\tr_i),
	\end{equation}
	where for each subtree $\tr_i$ ($i\in \{1,\dots, u\}$) of $\tr$,
	\begin{equation}
	\Psi(\tr_i) := \left\{
	\begin{aligned}
	& e, &\tr_i &= \phi,\\
	& Z^{(l)} \prod_{j=1}^{v}\Psi(\bar{\tr}_j), &\tr_i &=  [\bar{\tr}_1,\dots,\bar{\tr}_v]_l,\, l\in\{0,\dots, m\},
	\end{aligned}\right.
	\end{equation}
	where $e$ is the $d$-dimensional column vector whose elements are all 1. Note that, $\Psi(\tr)$ is defined recursively, and the product between vectors means component-wise multiplication.
\end{definition}
For the approximation $\{y_n\}$ calculated by the SRK method \eqref{e:srk}, we can also get an expansion with $I(\tr)$ and $F(\tr)(y_0)$ similar to the exact solution $X(t_n)$ \cite{Roßler2010}. Throughout the paper, equidistant time-step $h$ will be used in time discretization $\{0=t_0,\dots,t_N=T\}$, and will be omitted in some cases. Let $y_k$ be the numerical approximation of system \eqref{e:sdeIto} at time $t_k$, $k=0,\dots,N$.
Then mean-square convergence is considered:
\begin{definition}
	The numerical solution $\{y_k\}$ is said to have mean-square order $p$ ($>0$) if 
	\begin{equation}
	\big(\E|y_k - X(t_k)|^2 \big)^{\frac{1}{2}} = O(h^p),
	\quad k = 0,\dots,N,
	\end{equation}
	where $\E(\cdot)$ denotes the expectation of a random variable.
\end{definition}

Given the preparations above, we will mainly make use of the following theorem which can be found in \cite{Roßler2010} in detail.

\begin{theorem}\label{thm:order}
	Let $p \in \frac{1}{2}\mathbb{N}_0$ and $f, g_j\in C^{\lceil p\rceil, 2p+1}([0,T]\times\R^d,\R^d)$ for $j=1,\dots,m$. Then the SRK method \eqref{e:srk} has mean-square order $p$ if the following conditions are fulfilled
	\begin{enumerate}
		\item for all $\tr\in \Tr$ with $\rho(\tr)\leq p$
		\begin{equation}
		I(\tr) = \Phi(\tr),\quad \text{P-a.s}\, ;
		\end{equation}
		\item for all $\tr\in \Tr$ with $\rho(\tr) = p + \frac{1}{2}$
		\begin{equation}
		\E\Big(I(\tr)\Big) = \E\Big( \Phi(\tr) \Big).
		\end{equation}
	\end{enumerate}
\end{theorem}

\section{Order conditions of SRK methods for SDEs with additive noise}\label{sec:order_add}
Concerning the specific problem we focus on, i.e., Hamiltonian systems with additive noise, we consider the following It\^o sense SDE
\begin{equation}\label{e:sdeAdd}
\left\{
\begin{aligned}
dX(t) &= f\big(t, X(t)\big)dt + \sum_{r=1}^m g_r\big(t\big) dW_r(t), \quad t\in [0, T],\\
X(0) &= X_0. 
\end{aligned}\right.
\end{equation}
Then the $s$-stage SRK methods for system \eqref{e:sdeAdd} with $m$ independent additive noises are given by
\begin{equation}\label{e:srk_add}
\left\{
\begin{alignedat}{3}
Y_i &= y_n & &{}+ h\sum_{j=1}^{s} a_{ij} f(t_n+c_j h, Y_j) + \sum_{r=1}^{m} I_r \sum_{j=1}^{s}b_{ij}g_r(t_n+\widehat{c}_j h) \\
	& & &{} + \sum_{r=1}^{m} \frac{I_{r0}}{h} \sum_{j=1}^{s}d_{ij}g_r(t_n+\widehat{c}_j h), \quad i=1,\dots,s,\\
y_{n+1} &= y_n & &{} + h\sum_{i=1}^{s} \alpha_{i} f(t_n+c_i h, Y_i) + \sum_{r=1}^{m} I_r \sum_{i=1}^{s}\beta_{i}g_r(t_n+\widehat{c}_i h) \\
	& & &{}+ \sum_{r=1}^{m} \frac{I_{r0}}{h} \sum_{i=1}^{s}\gamma_{i}g_r(t_n+\widehat{c}_i h),
\end{alignedat}\right.
\end{equation}
where we denote $(A)_{s\times s}=(a_{ij})$, $(B)_{s\times s}=(b_{ij})$, $(D)_{s\times s}=(d_{ij})$, $(\alpha)_{s\times 1}=(\alpha_{i})$, $(\beta)_{s\times 1}=(\beta_{i})$, $(c)_{s\times 1}=(c_i)$, $(\widehat{c})_{s\times 1}=(\widehat{c}_i)$, then the order conditions are based on these coefficients accordingly.
These kinds of SRK methods \eqref{e:srk_add} can be characterized by an extended Butcher tableau as Table \ref{tab:butcher}.

\begin{table}[h]
	\centering
	\caption{Butcher tableau for SRK methods \eqref{e:srk_add}}
	\label{tab:butcher}
	\tabulinesep=5pt
	\large
	\begin{tabu}{c|c|c|c|c}
		$c$ & $A$        & $B$       & $D$        & $\widehat{c}$ \\ \hline
		& $\alpha^\top$ & $\beta^\top$ & $\gamma^\top$ &          
	\end{tabu}
\end{table}

Based on the colored rooted tree theory in Section \ref{sec:rooted_tree}, we are able to obtain a set of order conditions guaranteeing that SRK methods \eqref{e:srk_add} obtain mean-square order 1.5 as detailed below.

\begin{theorem}\label{thm:add_rk}
	Suppose that SDE \eqref{e:sdeAdd} with $m$ independent additive noises is approximated by SRK methods \eqref{e:srk_add}. Let $f\in C^{1,3}([0,T]\times \R^d, \R^d)$ and $g_j\in C^{1}([0,T]\times \R^d)$ for $j=1,\dots,m$.
	If the coefficients of SRK methods \eqref{e:srk_add} satisfy conditions
	\begin{enumerate}
		\item $c=Ae$,
		\item $\alpha^\top e = 1$,
		\item $\beta^\top e=1$,
		\item $\gamma^\top e=0$,
	\end{enumerate}
	then they are of mean-square order 1.0. If in addition conditions
	\begin{enumerate}
		\setcounter{enumi}{4}
		\item $\alpha^\top Ae=\frac{1}{2}$,
		\item $\alpha^\top Be=0$,
		\item $\alpha^\top De=1$,
		\item $\beta^\top \widehat{c}=1$,
		\item $\gamma^\top \widehat{c}=-1$,
		\item $\alpha^\top \left((Be)^2 + \frac{(De)^2}{3} + (Be)\cdot(De)\right) = \frac{1}{2}$,
	\end{enumerate}
	are fulfilled, then the mean-square order of SRK methods \eqref{e:srk_add} equals 1.5. 
\end{theorem}
\begin{proof}
	The proof here is similar to \cite[Appendix D]{Roßler2010}. Because the noise terms are additive,  we find that a elementary differential vanishes if its colored rooted tree contains a node following a stochastic node directly except if the deterministic node $\tau_0$ is the only succeeding end node. Table \ref{tab:add_condition} consists of colored rooted trees whose elementary differentials are non-zero. The special one is tree 4 for which we replace the $Aeh$ by $\widehat{c}h$ in $\Phi(\tr)$.
	Now we can list the following conditions under which the SRK methods \eqref{e:srk_add} obtain mean-square order 1.5.
	\begin{table}[htpb]
		\centering
		\caption{Colored rooted trees for \eqref{e:srk_add} with order less than or equal to 2}
		\label{tab:add_condition}
		\vspace{5pt}
		\begin{tabular*}{\textwidth}{@{\extracolsep{\fill}}@{}ccccc@{}}
			\toprule
			No. & $\tr$ & $\rho(\tr)$ & $I(\tr)$   & $\Phi(\tr)$                                                                 \\ \midrule
			1   & \makecell[c]{\tone}  & 0.5       & $I_i$    & $\left(\beta^\top I_i + \gamma^\top \frac{I_{i0}}{h}\right)e$                   \\ %\midrule
			2   & \makecell[c]{\ttwo}  & 1         & $I_0$    & $h \alpha^\top e$                                                            \\ %\midrule
			3   & \makecell[c]{\tthree}  & 1.5       & $I_{i0}$ & $h\alpha^\top\left(Be I_i + De\frac{I_{i0}}{h}\right)$                        \\ %\midrule
			4   & \makecell[c]{\tfour}  & 1.5       & $I_{0i}$ & $\left(\beta^\top I_i + \gamma^\top \frac{I_{i0}}{h}\right)\widehat{c}h$        \\ %\midrule
			5   & \makecell[c]{\tfive}  & 2         & $I_{00}$ & $h^2 \alpha^\top Ae$                                                         \\ %\midrule
			6   & \makecell[c]{\tsix}  & 2         & $\int_{0}^{h}W_i(s)W_j(s) ds$       & $h\alpha^\top (Be I_i + De\frac{I_{i0}}{h}) \cdot (Be I_j + De\frac{I_{j0}}{h})$ \\ %\midrule
			7   & \makecell[c]{\tseven}  & 2         & $\int_{0}^{h}W_i(s)^2 ds$       & $h\alpha^\top \left(Be I_i + De\frac{I_{i0}}{h}\right)^2$                      \\ \bottomrule
		\end{tabular*}
	\end{table}
	
	By properties of the It\^o integral, we have these facts for the increments:
	\begin{equation}\label{e:e_increments}
	\begin{aligned}
	\E(I_i)&=\E(I_{i0})=0,\quad \E(I_i^2)=h,\quad \E(I_{i0}^2)=\frac{h^3}{3}, \\
	\E(I_i I_{i0})&=\frac{1}{2}h^2,\quad\quad\quad \E(I_i I_{j0})= 0 \quad\text{for}\quad i\neq j.
	\end{aligned}
	\end{equation}
	
	In order to attain mean-square order 1.5, we need $I(\tr) = \Phi(\tr)$ for trees 1--4 ($\rho(\tr)\leq 1.5$), and $\E\big(I(\tr)\big) = \E\big(\Phi(\tr)\big)$ for trees 5--7  ($\rho(\tr)= 2$) respectively in Table \ref{tab:add_condition} due to Theorem \ref{thm:order}.
	
	To be specific, for tree 1,
	\[
	I_i = \left(\beta^\top I_i + \gamma^\top \frac{I_{i0}}{h}\right)e \quad \Rightarrow\quad \beta^\top e=1,\quad \gamma^\top e =0. 
	\]
	For tree 2,
	\[
	I_0 = h \alpha^\top e \quad \Rightarrow\quad \alpha^\top e =1.
	\]
	For tree 3,
	\[
	I_{i0} = h\alpha^\top\left(Be I_1 + De\frac{I_{i0}}{h}\right)\quad \Rightarrow\quad \alpha^\top Be = 0, \quad \alpha^\top De =1.
	\]
	For tree 4,
	\[
	I_{0i} = \left(\beta^\top I_i + \gamma^\top \frac{I_{i0}}{h}\right)\widehat{c}h \quad \Rightarrow\quad \beta^\top \widehat{c}=1, \quad \gamma^\top \widehat{c}=-1.
	\]
	For tree 5,
	\[
	I_{00} = h^2 \alpha^\top Ae \quad \Rightarrow\quad \alpha^\top Ae = \frac{1}{2}.
	\]
	For tree 6,
	\[
	\E\left(\int_{0}^{h}W_i(s)W_j(s) ds\right) = \E\left(h\alpha^\top (Be I_i + D\frac{I_{i0}}{h}) \cdot (Be I_j + D\frac{I_{j0}}{h}\right) \, \Rightarrow\, 0 = 0.
	\]
	For tree 7,
	\begin{multline*}
	\E\left(\int_{0}^{h}W_i(s)^2 ds\right) = \frac{1}{2}h^2 = \E\left(h\alpha^\top \left(Be I_i + De\frac{I_{i0}}{h}\right)^2\right) \\
	\Rightarrow\quad \alpha^\top \left((Be)^2 + \frac{(De)^2}{3} + (Be)\cdot(De)\right) = \frac{1}{2}
	\end{multline*}
	
	Adding the usual condition $Ae = c$ used in the construction of deterministic RK methods, the conclusion follows immediately from Theorem \ref{thm:order}. 
\end{proof}

\begin{remark}
	The conditions for the first mean-square order SRK methods (conditions 1--4 in Theorem \ref{thm:add_rk}) agree with the results in \cite{Hong2015,Ma2012} when considering the additive noise case.
\end{remark}

Notice that conditions 1, 2, 5 in this theorem are the usual order conditions for normal deterministic RK methods. Therefore we can choose some deterministic RK methods as the base model and select the additional coefficients to fulfill Theorem \ref{thm:add_rk}. To meet these order conditions, the least number of stages we need is 2. For example, considering the explicit 2-stage SRK method, namely the matrix $A$ only has one non-zero coefficient $a_{21}$, then we have 9 equations with 17 coefficients to determine according to Theorem \ref{thm:add_rk}. If additionally we assume that $\widehat{c}_1=0, \widehat{c}_2=1$, then we can  get a class of methods SRK-$\alpha_1$ ($0<\alpha_1<1$ is free) with coefficients in Table \ref{tab:explicit_2s}.
\begin{table}[h]
	\centering
	\caption{Butcher tableau of 2-stage explicit SRK-$\alpha_1$ methods ($0<\alpha_1<1$ is free)}
	\label{tab:explicit_2s}
	\tabulinesep=5pt
	\vspace{5pt}
	\begin{tabu}{c|cc|cc|cc|c}
		0 & 0              & 0              & $-\sqrt{\frac{2(1-\alpha_1)}{3\alpha_1}}$ &  0 & $1+\sqrt{\frac{3(1-\alpha_1)}{2\alpha_1}}$ &  0  & 0  \\
		$\frac{1}{2-2\alpha_1}$ & $\frac{1}{2-2\alpha_1}$           & 0             &  $\sqrt{\frac{2\alpha_1}{3(1-\alpha_1)}}$ &  0 & $1+\frac{ \sqrt{6\alpha_1(1-\alpha_1)}}{2(\alpha_1-1)}$ & 0   & 1 \\ \hline
		& $\alpha_1$ & $1-\alpha_1$ & 0                    & 1 & 1 & $-1$ & 
	\end{tabu}
\end{table}

For instance, choosing $\alpha_1=\frac{1}{2}$ leads to a $2$-stage explicit SRK scheme which is of order 2 in the deterministic part (Euler-Heun), and  mean-square order 1.5 in the stochastic case. We list the coefficients of this scheme in Table \ref{tab:explicit} and call it SRK-0.5 in the sequel.

\begin{table}[h]
	\centering
	\caption{Butcher tableau of SRK-0.5 scheme ($\alpha_1=\frac{1}{2}$ in Table \ref{tab:explicit_2s})}
	\label{tab:explicit}
	\tabulinesep=5pt
	\vspace{5pt}
	\begin{tabu}{c|cc|cc|cc|c}
		0 & 0              & 0              & $-\sqrt{\frac{2}{3}}$ & 0  & $1+\sqrt{\frac{3}{2}}$ & 0   &  0 \\
		1 & 1           &  0            & $\sqrt{\frac{2}{3}}$ &  0 & $2-\frac{2+\sqrt{6}}{2}$ & 0   & 1 \\ \hline
		& $\frac{1}{2}$ & $\frac{1}{2}$ & 0                    & 1 & 1 & $-1$ & 0
	\end{tabu}
\end{table}

Meanwhile, we can acquire some other SRK methods with higher deterministic order provided we increase the number of stages, which leads to more coefficients for us to determine. We may also observe that, the form of the SRK methods \eqref{e:srk_add} we construct here demands $2m$ stochastic increments at one time step, namely $I_i$ and $I_{i0}$, $i=1,\dots,m$. The procedure of generating them and the technique of checking the corresponding mean-square order will be investigated in Section \ref{sec:num} in detail.

\section{Symplectic conditions of SRK methods for stochastic Hamiltonian systems with additive noise}\label{sec:sym_condition}
In this section we present symplectic conditions for SRK methods applied to stochastic Hamiltonian systems with additive noise, and combine the order conditions obtained in Section \ref{sec:order_add} to construct possible SSRK methods.
It is known that the standard $s$-stage symplectic RK methods for general ODEs must be implicit and have coefficients satisfying the following conditions \cite{Cooper1987,Sanz1988} (details can be found in monographs \cite{Hairer2006,Feng2010book,Sanz1994NumericalHamiltonian})
\begin{equation}\label{e:symplecticRK}
\alpha_i a_{ij} + \alpha_j a_{ji} = \alpha_i \alpha_j, \quad \text{for}\,\, i, j = 1,\dots,s.
\end{equation}

For Hamiltonian systems with additive noise, the conditions for SRK methods \eqref{e:srk_add} conserving symplecticity are just the same as that for deterministic symplectic RK methods.

\begin{theorem}\label{thm:symplectic_cond}
	The SRK methods \eqref{e:srk_add} for Hamiltonian systems with additive noise are symplectic if conditions \eqref{e:symplecticRK} are satisfied.
\end{theorem}
\begin{proof}
	Owing to the additive noise in our model, we can rewrite the SRK method \eqref{e:srk_add} as
	\begin{equation}\label{e:srk_add_proof}
		\left\{
		\begin{aligned}
		P_i = p_n + h\sum_{j=1}^{s} a_{ij} f(t_n+c_j h, P_j, Q_j) + \xi_i,\\
		Q_i = q_n + h\sum_{j=1}^{s} a_{ij} f(t_n+c_j h, P_j, Q_j) + \eta_i,\\
		p_{n+1} = p_n + h\sum_{i=1}^{s} \alpha_{i} f(t_n+c_i h, P_i, Q_i) + \phi,\\
		q_{n+1} = q_n + h\sum_{i=1}^{s} \alpha_{i} f(t_n+c_i h, P_i, Q_i) + \psi,
		\end{aligned}\right.
	\end{equation}
	where $\xi_i, \eta_i, \phi, \psi$ are random variables which are independent of $P_i$ and $Q_i$ at each step. Thus we are able to attain the symplectic condition $dp_{n+1}\wedge dq_{n+1} = dp_n\wedge dq_n$ for \eqref{e:srk_add_proof} if \eqref{e:symplecticRK} are satisfied by the simple calculation of wedge product \cite{Milstein2002add}.
	
\end{proof}

\begin{remark}
	Symplectic conditions for Hamiltonian systems with multiplicative noise can be found in \cite{Ma2012,Hong2015}. However, more conditions similar to \eqref{e:symplecticRK} must be added.
\end{remark}

Now we construct SSRK methods of mean-square order 1.5 for Hamiltonian systems with additive noise. Here we just consider the case of lower stage SRK methods. As before the 1-stage method is not suitable here. Therefore, we select the family of 2-stage order 2.0 symplectic RK methods in the deterministic case \cite{Hairer2006} with coefficients in Table \ref{tab:symRK}, where the parameter $\alpha_1\in (0,1)$ is free to choose.

\begin{table}[h]
	\centering
	\caption{Butcher tableau of the 2-stage order 2.0 diagonal implicit symplectic RK methods}
	\label{tab:symRK}
	\tabulinesep=5pt
	\vspace{5pt}
	\begin{tabu}{c|cc}
		$\frac{\alpha_1}{2}$ & $\frac{\alpha_1}{2}$ & 0  \\
		$\frac{\alpha_1 + 1}{2}$& $\alpha_1$ & $\frac{1 - \alpha_1}{2}$ \\ \hline
		& $\alpha_1$ & $1 - \alpha_1$
	\end{tabu}
\end{table}

Concerning the conditions in Theorem \ref{thm:add_rk}, we gain the conditions below:
\begin{align*}
	\beta^\top e = 1 \quad &\Rightarrow\quad \beta_1 +\beta_2 = 1,\\
	\gamma^\top e = 0 \quad &\Rightarrow\quad \gamma_1 + \gamma_2 = 0,\\
	\beta^\top \widehat{c} = 1 \quad & \Rightarrow\quad \beta_1 \widehat{c}_1 + \beta_2\widehat{c}_2 = 1,\\
	\gamma^\top \widehat{c} = -1 \quad & \Rightarrow\quad \gamma_1 \widehat{c}_1 + \gamma_2\widehat{c}_2 = -1,\\
	\alpha^\top Be = 0 \quad & \Rightarrow\quad \alpha_1 b_1 + (1-\alpha_1) b_2 = 0,\\
	\alpha^\top De = 1 \quad & \Rightarrow\quad \alpha_1 d_1 + (1-\alpha_1) d_2 = 1,
\end{align*}
\begin{multline*}
	\alpha^\top \left((Be)^2 + \frac{(De)^2}{3} + (Be)\cdot(De)\right) = \frac{1}{2} \quad  \Rightarrow\quad \\ \alpha_1\left(b_1^2 + \frac{d_1^2}{3} + b_1 d_1\right) + (1-\alpha_1)\left(b_2^2 + \frac{d_2^2}{3} + b_2 d_2\right) = \frac{1}{2},
\end{multline*}
where we denote $Be=(b_1, b_2)^\top$, and $De=(d_1, d_2)^\top$.

Let $\widehat{c}_1=0$, $\widehat{c}_2=1$. We thus obtain a one-parameter family SSRK-$\alpha_1$ and a two-parameter family SSRK-$\alpha_1$-$b_1$ listed in Table \ref{tab:SSRKa} and \ref{tab:SSRKab} respectively, which are 2-stage diagonal implicit SSRK methods with mean-square order 1.5. For example, if we choose $\alpha_1=\frac{1}{2}$ in Table \ref{tab:SSRKa}, then we get the SSRK-0.5 in Table \ref{tab:SSRK-0.5}, which is similar to the scheme proposed in \cite{Milstein2002add} but totally derivative-free here.
\begin{table}[h]
	\centering
	\caption{Butcher tableau of the 2-stage diagonal implicit SSRK$-\alpha_1$ method ($0<\alpha_1<1$ is free)}
	\label{tab:SSRKa}
	\tabulinesep=5pt
	\vspace{5pt}
	\begin{tabu}{c|cc|cc|cc|c}
		$\frac{\alpha_1}{2}$ & $\frac{\alpha_1}{2}$              & 0              & $-\sqrt{\frac{2(1-\alpha_1)}{3\alpha_1}}$ &  0 & $1+\sqrt{\frac{3(1-\alpha_1)}{2\alpha_1}}$ &  0  & 0  \\
		$\frac{\alpha_1+1}{2}$ & $\alpha_1$           & $\frac{1-\alpha_1}{2}$             &  $\sqrt{\frac{2\alpha_1}{3(1-\alpha_1)}}$ &  0 & $1+\frac{ \sqrt{6\alpha_1(1-\alpha_1)}}{2(\alpha_1-1)}$ & 0   & 1 \\ \hline
		& $\alpha_1$ & $1-\alpha_1$ & 0                    & 1 & 1 & $-1$ & 
	\end{tabu}
\end{table}

\begin{table}[h]
	\centering
	\caption{Butcher tableau of the 2-stage diagonal implicit SSRK-$\alpha_1$-$b_1$ method ($0<\alpha_1<1$ and $-\frac{2}{3}\sqrt{\frac{1-\alpha_1}{\alpha_1}}<b_1<\frac{2}{3}\sqrt{\frac{1-\alpha_1}{\alpha_1}}$ are free parameters)}
	\label{tab:SSRKab}
	\tabulinesep=5pt
	\vspace{5pt}
	\begin{tabu}{c|cc|cc|cc|c}
		$\frac{\alpha_1}{2}$  & $\frac{\alpha_1}{2}$ &           0            & $b_1$ & 0 &       $1-\frac{3b_1}{2}-\frac{1}{2}\sqrt{\frac{2}{\alpha_1}-3b_1^2-2}$        & 0  & 0 \\
		$\frac{\alpha_1+1}{2}$ &      $\alpha_1$      & $\frac{1-\alpha_1}{2}$ & $\frac{\alpha_1 b_1}{\alpha_1-1}$  & 0 & $1- \frac{ 3b_1\alpha_1+\alpha_1\sqrt{\frac{2}{\alpha_1}-3b_1^2-2}}{2(\alpha_1-1)}$ & 0  & 1 \\ \hline
		&      $\alpha_1$      &      $1-\alpha_1$      &                     0                     & 1 &                            1                            & $-1$ &
	\end{tabu}
\end{table}

\begin{table}[h]
	\centering
	\caption{Butcher tableau of the SSRK-0.5 scheme ($\alpha_1 = \frac{1}{2}$ in Table \ref{tab:SSRKa})}
	\label{tab:SSRK-0.5}
	\tabulinesep=5pt
	\vspace{5pt}
	\begin{tabu}{c|cc|cc|cc|c}
		$\frac{1}{4}$ & $\frac{1}{4}$ &       0       & $-\sqrt{\frac{2}{3}}$ & 0 &  $1+\sqrt{\frac{3}{2}}$  & 0  & 0 \\
		$\frac{3}{4}$ & $\frac{1}{2}$ & $\frac{1}{4}$ & $\sqrt{\frac{2}{3}}$  & 0 & $1-\sqrt{\frac{3}{2}}$ & 0  & 1 \\ \hline
		& $\frac{1}{2}$ & $\frac{1}{2}$ &           0           & 1 &            1             & $-1$ &
	\end{tabu}
\end{table}

\section{Second-order Hamiltonian systems with additive noise}\label{sec:2order_add}
In this section, we focus on the special second-order Hamiltonian systems with additive noise of the following form \cite{Milstein2002add}:
\begin{equation}\label{e:2order1}
	M\ddot{X} + \nabla U(t, X) = \sum_{r=1}^{m}\sigma_r (t) \dot{W}_r,
\end{equation}
where $M$ is a symmetric and invertible $n\times n$ constant matrix, $U$ is a differentiable function. These systems arise in a wide range of fields such as classical mechanics, molecular dynamics, biology and quantum mechanics when considering the random force effect \cite{Gitterman2005}.  Let
\begin{equation}
	Q(t) = X(t), \quad P(t) = M\dot{Q}(t),
\end{equation}
then \eqref{e:2order1} can be regarded as a $2n$-dimensional Hamiltonian system with additive noise:
\begin{equation}\label{e:order2_hamilton}
	\left\{
	\begin{aligned}
		dP &= -\nabla U(t,Q) + \sum_{r=1}^{m} \sigma_r(t) dW_r(t),\\
	dQ &= M^{-1}P dt.
	\end{aligned}\right.
\end{equation}
Thus, the corresponding Hamiltonian functions of \eqref{e:order2_hamilton} are
\begin{equation}
	H_0(t,p,q) = \frac{1}{2} p^\top M^{-1}p + U(t, q),
\end{equation}
\begin{equation}
	H_r(t,p,q) = - \sum_{r=1}^{m} \sigma_r(t) q, \quad r=1,\dots, m
\end{equation}
Obviously this is a special example of Hamiltonian systems with additive noise, and the mean-square order conditions of the foregoing SSRK methods are also available here. However, according to the specific features of \eqref{e:order2_hamilton}, we are able to simplify some conditions and get higher mean-square order without more effort.

\begin{theorem}\label{thm:2order_add}
	Let $\nabla U\in C^{1,3}([0,T]\times \R^d, \R^d)$ and $\sigma_j\in C^{1}([0,T], \R)$ for $j=1,\dots,m$. If the SRK methods \eqref{e:srk_add} for the second-order Hamiltonian systems \eqref{e:order2_hamilton} possess coefficients satisfying
	\begin{enumerate}
		\item $c=Ae$,
		\item $\alpha^\top e = 1$,
		\item $\beta^\top e=1$,
		\item $\gamma^\top e=0$,
	\end{enumerate}
	then they are of mean-square order 1.0. If in addition conditions
	\begin{enumerate}
		\setcounter{enumi}{4}
		\item $\alpha^\top Ae=\frac{1}{2}$,
		\item $\alpha^\top Be=0$,
		\item $\alpha^\top De=1$,
		\item $\beta^\top \widehat{c}=1$,
		\item $\gamma^\top \widehat{c}=-1$,
	\end{enumerate}
	are satisfied, then they have mean-square order 2.0.
\end{theorem}
\begin{proof}
	We find that  conditions of the theorem above are similar to Theorem \ref{thm:add_rk} except that condition 10 in Theorem \ref{thm:add_rk} is unnecessary here. In addition, because of the special structure of system \eqref{e:order2_hamilton}, we obtain mean-square order 2.0 here which is higher than that of Theorem \ref{thm:add_rk}. In order to reach mean-square order 2.0, we need to analyze three additional colored rooted trees ($\rho(\tr)=2.5$) listed in Table \ref{tab:add_order2_condition}.
	\begin{table}[h]
		\centering
		\caption{Colored rooted tree for \eqref{e:srk_add} with order equal to 2.5}
		\label{tab:add_order2_condition}
		\vspace{5pt}
		\begin{tabular*}{\textwidth}{@{\extracolsep{\fill}}@{}ccccc@{}}
			\toprule
			No. & $\tr$ & $\rho(\tr)$ & $I(\tr)$   & $\Phi(\tr)$                                                                 \\ \midrule
			8   & \makecell[c]{\teight}  & 2.5       &  $I_{i00}$   & $h\alpha^\top (hA) \left(Be I_1 + De\frac{I_{i0}}{h}\right)$                  \\ %\midrule
			9   & \makecell[c]{\tnine}  & 2.5         &  $\int_{0}^{h}s\cdot W_s ds$   & $h \alpha^\top (hAe)\cdot\left(Be I_1 + De\frac{I_{i0}}{h}\right) $                                                            \\ %\midrule
			10   & \makecell[c]{\tten}  & 2.5       & $I_{0i0}$ & $h\alpha^\top\left(B I_1 + D\frac{I_{i0}}{h}\right)(h\widehat{c})$\\
			\bottomrule
		\end{tabular*}
	\end{table}
	
	The condition 10 in Theorem \ref{thm:add_rk} is derived from the tree 7 in Table \ref{tab:add_condition}, which is essential for general system \eqref{e:srk_add}. However, for system \ref{e:order2_hamilton}, the corresponding elemental differential of tree 7 is
	\begin{equation}
		F(\tr_7)(x) = f^{(2)}(x)\Big(F(\tau_i)(x),F(\tau_i)(x)\Big) = 0,
	\end{equation}
	because the $n$th component of it is
	\begin{equation}
		\begin{aligned}
			\Big(F(\tr_7)(x)\Big)_n &=  \sum_{j_1,j_2 =1}^{d} \frac{\pa^2 (f(x)_n)}{\pa x^{j_1} \pa x^{j_2}} \Big(\big(g_i(x)\big)_{j_1}, \big(g_i(x)\big)_{j_2}\Big)\\
			&= \left\{
			\begin{aligned}
				& \sum_{j_1,j_2 =1}^{d} \frac{\pa^2 \big(\nabla U(q)\big)_n}{\pa q^{j_1} \pa q^{j_2}}\big(0, 0\big) & \text{for} & & n &= 1,\dots ,\frac{d}{2},\\
				&\sum_{j_1,j_2 =1}^{d} \frac{\pa^2 \big(M^{-1}p\big)_n}{\pa p^{j_1} \pa p^{j_2}}\big(\sigma_{j_1}, \sigma_{j_2}\big) & \text{for} & & n &= \frac{d}{2}+1,\dots ,d,
			\end{aligned}\right.\\
			&= 0,
		\end{aligned}
	\end{equation}
	where $\big(f(x)\big)_n$ means the $n$th element of the vector function $f(x)$.
	Thus tree 7 is unnecessary in this situation. With the same analysis, the elementary differential of tree 6 in Table \ref{tab:add_condition} vanishes as well. Moreover, for the additional trees 8, 9, 10 in Table \ref{tab:add_order2_condition}, we can  check that
	\begin{equation}
		\E\big(I(\tr)\big) = \E\big(\Phi(\tr)\big).
	\end{equation}
	In short, we have 
	\begin{equation}
		\begin{aligned}
			I(\tr) &= \Phi(\tr) && \text{for} & \rho(\tr) &\leq 2,\\
			\E\big(I(\tr)\big) &= \E\big(\Phi(\tr)\big) && \text{for} & \rho(\tr) &= 2.5.
		\end{aligned}
	\end{equation}
	Then applying Theorem \ref{thm:order} completes the proof. 
\end{proof}

\begin{remark}
	If the $dQ$ term in \eqref{e:order2_hamilton} contains additive noise as well, we can also get similar conditions to obtain mean-square order 2.0 SRK methods, but in this case the condition 10 in Theorem \ref{thm:add_rk} must be added.
\end{remark}

\begin{theorem}
	The SSRK-$\alpha_1$ and SSRK-$\alpha_1$-$b_1$ methods with coefficients for system \eqref{e:order2_hamilton} listed in Table \ref{tab:SSRKa} and \ref{tab:SSRKab} respectively are symplectic and of mean-square order 2.0.
\end{theorem}
\begin{proof}
	The coefficients of SSRK-$\alpha_1$ and SSRK-$\alpha_1$-$b_1$ in Table \ref{tab:SSRKa} and \ref{tab:SSRKab} satisfy Theorem \ref{thm:symplectic_cond} and \ref{thm:2order_add} above clearly; thus it is symplectic and of mean-square order 2.0. 
\end{proof}

\section{Numerical experiments}\label{sec:num}

In this section, we perform numerical tests to verify the mean-square convergence  order and geometric superiority of our numerical schemes proposed in Section \ref{sec:order_add} and \ref{sec:sym_condition}. To reach high mean-square order, the schemes must contain some more stochastic increments besides $I_i = \Delta W_i$ in the Euler-Maruyama scheme \cite{Kloeden1992} at each step, which are actually the multiple stochastic integrals appearing in the It\^o-Taylor expansion (cf. \cite{Kloeden1992}). Generally there is no simple way to simulate these multiple It\^o integrals exactly and effectively \cite{Wiktorsson2001}. Nevertheless, owing to the special structure of SDEs we consider here, the Hamiltonian systems with additive noise, we observe that the schemes we derive here will be just in need of another stochastic increment $I_{i0} = \int_{t_n}^{t_{n+1}} \int_{t_n}^{s} dW_i(s) \,ds$ for each Wiener process $W_i$ on every step.

By using the fact that $I_i$ and $I_{i0}$ are two centered Gaussian random variables, they can be simulated by two independent standard Gaussian random variables $U_1$ and $U_2$ \cite{Milstein2004}. Specifically,
\begin{equation}
I_i := \sqrt{h} U_1,\quad I_{i0} := \frac{1}{2}h^{3/2}\left(U_1 + \frac{U_2}{\sqrt{3}}\right),
\end{equation}
which imply that we demand $2m$ independent Gaussian random variables at each time step.
To check the mean-square order of convergence for our schemes, we may test the mean-square errors at the terminal time $T$ according to different time step-size $h$. It might also be noted that with the change of $h$, the stochastic increments we use in the numerical schemes have to be in the same sample path respectively \cite{Higham2001}. For technical details please see e.g. \cite{Burrage1999runge}.

%For instance, if we have generated the  increments of step-size $h$, then we are able to generate the larger step increments corresponding to $h' = 2h$ as follows.

%If we denote the increments at time $t_n$ by $\Delta W_{t_n}^h$ and $\Delta Z_{t_n}^h$ (the subscript $i$ is omitted), which in fact stand for
%\begin{equation}
%	\Delta W_{t_n}^h = W(t_{n+1}) - W(t_n), \quad \Delta Z_{t_n}^h = \int_{t_n}^{t_{n+1}}\int_{t_n}^{s} dW(s_2) \,ds,
%\end{equation}
%then the new increments with time step-size $2h$ are
%\begin{equation}
%	\begin{aligned}
%		\Delta W_{t_n}^{2h} &= W(t_{n+2}) - W(t_n)\\
%		&= \Delta W_{t_n}^h + \Delta W_{t_{n+1}}^h,
%	\end{aligned}
%\end{equation}
%
%\begin{equation}\label{e:dZ}
%	\begin{aligned}
%		\Delta Z_{t_n}^{2h} &= \int_{t_n}^{t_{n+2}}\int_{t_n}^{s} dW(s_2) \, ds\\
%		&= \Delta Z_{t_n}^h + \Delta Z_{t_{n+1}}^h + h\Delta W_{t_n}^{h},
%	\end{aligned}
%\end{equation}
%where the third term $h\Delta W_{t_n}^{h}$ in \eqref{e:dZ} is essential and is easily overlooked. Thus, when checking the mean-square order of convergence for the numerical schemes in Section \ref{sec:order_add} numerically, we have to iteratively generate the stochastic increments layer by layer. Namely, we coarsen every simple path of a Wiener process from the discretization with a smaller step-size $h$ to the one with $2h$.

\subsection{Stochastic harmonic oscillator with additive noise}
Here we consider the stochastic harmonic oscillator with scalar additive noise given by:
\begin{equation}\label{e:linear_oscillator}
\left\{
\begin{aligned}
dP &= -Qdt + \sigma dW(t), & P(0) = p_0, \\
dQ &= \phantom{-}Pdt, & Q(0) = q_0,
\end{aligned}\right.
\end{equation}
where $Q$ is the position and $P$ is the velocity of a particle under the simple harmonic restoring force and a random white noise force with intensity $\sigma$ \cite{Milstein2004}. Also, system \eqref{e:linear_oscillator} is a simple example of system \eqref{e:order2_hamilton} in autonomous case (drift and diffusion term are independent of t), where $M=1$ and $U(t, q) = \frac{1}{2}q^2$. Thus, the Hamiltonians of \eqref{e:linear_oscillator} are $H_0(p,q) = \frac{1}{2}(p^2 + q^2)$ and $H_1(p,q) = -\sigma q$. We also note that
\begin{equation}\label{e:linear_growth}
\E\Big(H_0\big(P(t),Q(t)\big)\Big) = \frac{1}{2}(p_0^2+q_0^2) + \frac{1}{2}\sigma^2 t,
\end{equation}
simply by using the It\^{o}'s formula, which is a significant geometric property of \eqref{e:linear_oscillator} and can be found in e.g. \cite{Burrage2012,Hong2007,Melbø2004} for detail. This implies that the expectation of $H_0$ along the exact solution of \eqref{e:linear_oscillator} (the second moment of the exact solution in this case) has linear growth. Recently,  \cite{Senosiain2014} gives a review on numerical schemes for solving this kind of linear stochastic oscillator, but it does not contain high mean-square order methods like SSRK-$\alpha_1$ in this paper.

\begin{table}[htbp]
	\centering
	\caption{Mean-square errors for \eqref{e:linear_oscillator} with different schemes} 
	\label{tab:linear_osc}
	\begin{tabular*}{1\textwidth}{@{\extracolsep{\fill}}@{}lcccccc@{}}
		\toprule
		$h$   & $2^{-1}$ & $2^{-2}$ & $2^{-3}$ & $2^{-4}$ & $2^{-5}$ & order \\ \midrule
		Euler & 3.71E-01 & 1.82E-01 & 8.72E-02 & 4.09E-02 & 1.82E-02 & 1.09  \\
		SRK-0.5  & 1.99E-01 & 4.82E-02 & 1.19E-02 & 3.00E-03 & 7.00E-04 & 2.03  \\
		SSRK-0.5  & 8.69E-02 & 2.11E-02 & 5.30E-03 & 1.30E-03 & 3.00E-04 & 2.04  \\ \bottomrule
	\end{tabular*}
\end{table}

Firstly, we check the mean-square convergence of SSRK-0.5 scheme in Section \ref{sec:sym_condition}. For comparison, the Euler-Maruyama and mean-square order 1.5 SRK-0.5 scheme proposed in Section \ref{sec:order_add} are also presented in this part. 
We simulate them at terminal time $T=1$, with $\sigma=1$, $(p_0, q_0)=(1,0)$ in system \eqref{e:linear_oscillator}. To avoid applying more random variables in simulating the exact solution, we just use order 2.0 strong Taylor type scheme \cite{Kloeden1992} with $h=2^{-14}$ as the reference solution $X_{T}^{ref}$, and the corresponding mean-square errors are computed as:
\begin{equation}
\frac{1}{M} \sum_{i=1}^{M} \Big|X_{N}(\omega_i)-X_{T}^{ref}(\omega_i)\Big|^2
\end{equation}
where $M=3000$ denotes the number of sample paths we simulate.
Table \ref{tab:linear_osc} shows the mean-square errors of Euler-Maruyama, SRK and SSRK methods, where the last column lists the convergence order calculated by method of the least square fitting \cite{Higham2001}. Moreover, Figure \ref{f:linearosc_order} shows them graphically. There are three dashed lines as references which have slopes 1.0, 1.5 and 2.0, respectively in order to demonstrate the mean-square convergence order for these methods. So the mean-square orders for SRK-0.5 and SSRK-0.5 are both 2.0. It is consistent with theoretical analysis in Theorem \ref{thm:2order_add}.

\begin{figure}[htb]
	\centering
	\includegraphics[width=1\textwidth]{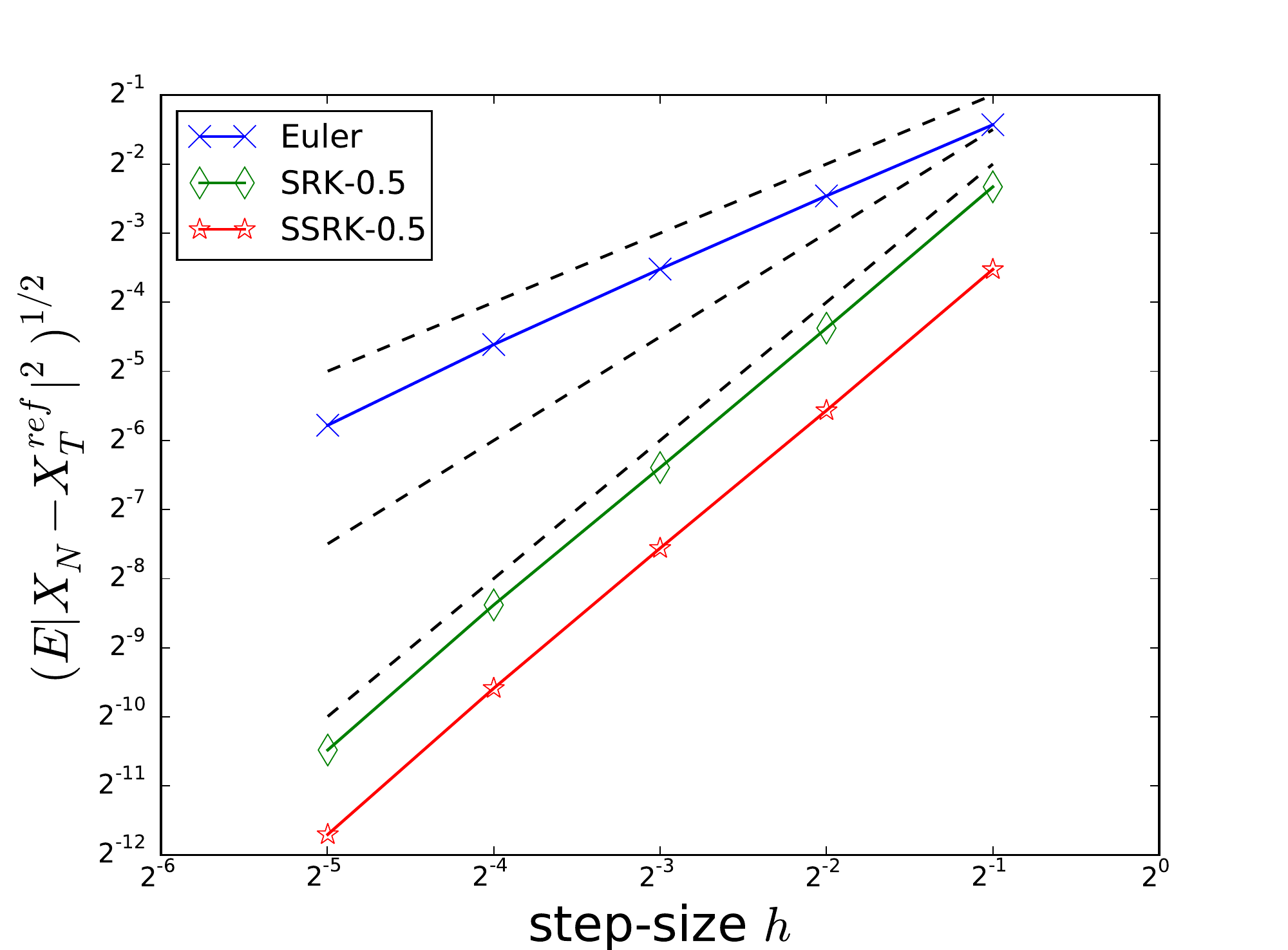}
	\caption{Mean-square errors at $T=1$ of  three methods for linear oscillator \eqref{e:linear_oscillator} with $\sigma = 1$, $(p_0 , q_0)=(1, 0)$. The dashed reference lines have slopes 1, 1.5 and 2 respectively.}\label{f:linearosc_order}
\end{figure}

Next, we consider the numerical property of the linear growth  \eqref{e:linear_growth} for SSRK-0.5 scheme, that is the second moment of numerical solutions over time. We also use $M=3000$ sample paths to simulate the expectation and the time-step is chosen to be $h = 0.1$. From \cite{Hong2007}, the Euler-Maruyama scheme produces solutions whose second moment grows exponentially fast as is shown in the left-hand side of Figure \ref{fig:energy_euler} directly, so the Euler-Maruyama scheme is unacceptable for this problem in long time simulation. 

\begin{figure}[htbp]
	\centering
	\includegraphics[width=1\textwidth]{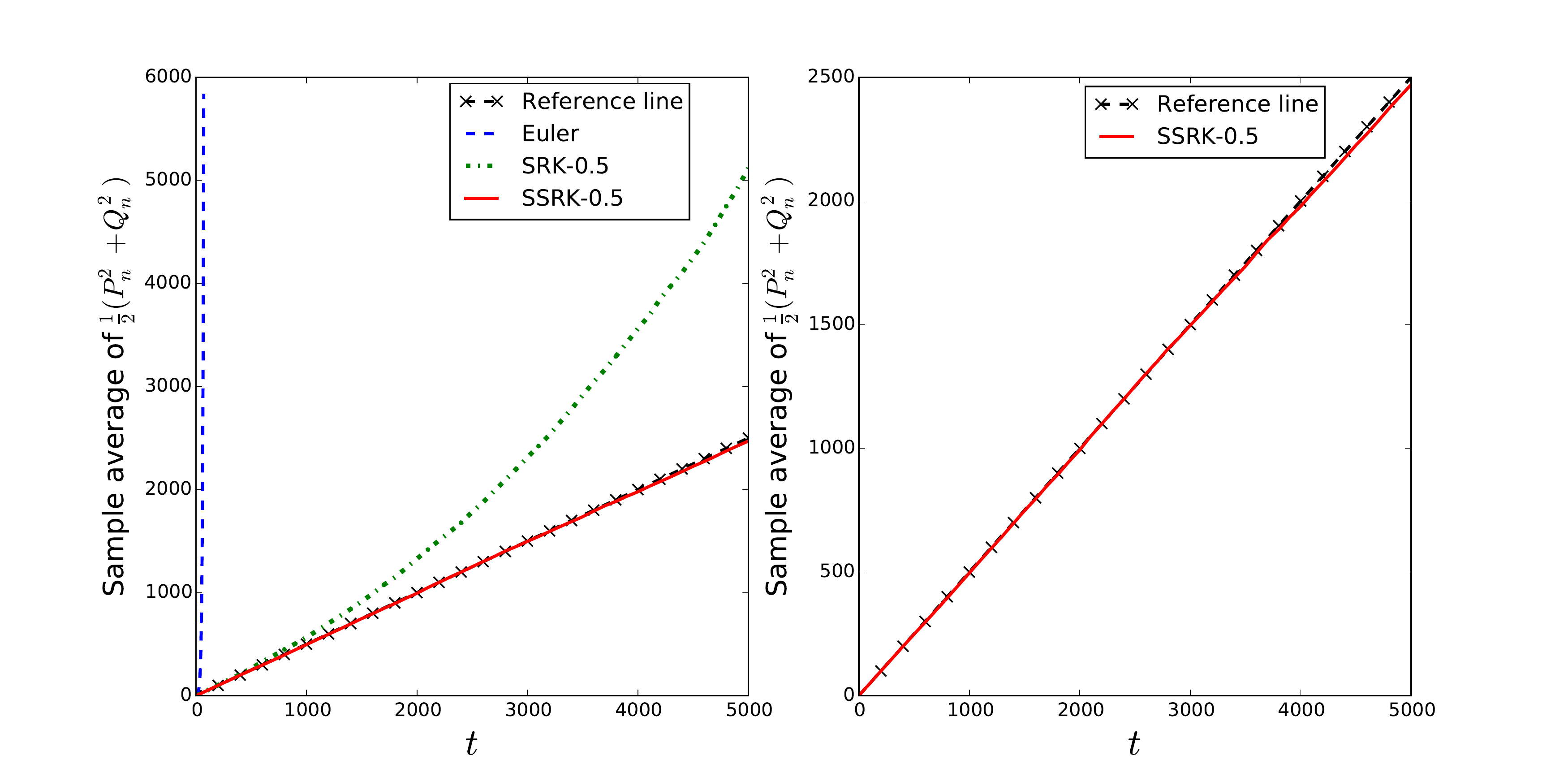}
	\caption{The growth rate for $E(H_0)$ along the numerical solutions with time interval $t\in[0,5000]$.}
	\label{fig:energy_euler}
\end{figure}

On the other hand, SRK-0.5 and SSRK-0.5 schemes behave much better in the long time run, which can be observed in Figure \ref{fig:energy_euler}. Moreover, it also displays that although the SRK-0.5 scheme is more stable than the Euler-Maruyama scheme in preserving this property, it also deviates from the original reference line after $t=1000$, but the SSRK-0.5 scheme coincides with the reference line much better than SRK-0.5 scheme. We solely test the second moment of numerical solutions created by the SSRK-0.5 scheme with larger time interval $t_n\in[0,5000]$ in right-hand side of Figure \ref{fig:energy_euler}, which shows that the SSRK-0.5 scheme preserves the linear growth property \eqref{e:order2_hamilton} with quite high accuracy.

Actually, using the fact in \eqref{e:e_increments},
the numerical solution arising from the SSRK-0.5 scheme for system \eqref{e:linear_oscillator} satisfies
\begin{equation}
\E\Big(H_0\big(P_n,Q_n\big)\Big) = \frac{1}{2}(p_0^2+q_0^2) + \frac{1}{2}\sigma^2 t_n + C(h)\sigma^2 t_n,
\end{equation}
where 
\begin{equation}\label{e:Ch_SSRK}
	C(h) = \frac{-16(4-\sqrt{6})h^2 + (4-\sqrt{6})h^4}{3(16+h^2)^2}.
\end{equation}
Compared with \eqref{e:linear_growth}, it is $C(h)$ \eqref{e:Ch_SSRK} that gives rise to the error in this linear growth property. However, it can be seen that $C(h) = O(h^2)$, as $h\rightarrow 0$. Thus, the expectation of $H_0$ along the numerical solution generated by SSRK-0.5 has also linear growth but the slope is slightly disturbed by $C(h)$.

%\begin{figure}[htbp]
%	\centering
%	\includegraphics[width=0.7\textwidth]{fig_p100_srk}
%	\caption{}
%	\label{fig:P100_SRK}
%\end{figure}

In order to minimize the error in the linear growth \eqref{e:linear_growth} for a fixed step-size $h$, we set $\alpha_1\in (0,1)$ in Table \ref{tab:SSRKa} as a free parameter, so we get a parametric scheme containing $\alpha_1$. Calculating the expectation of $H_0(\cdot) $ along the numerical solution by this scheme leads to
\begin{equation}
	\E\Big(H_0\big(P_n,Q_n\big)\Big) = \frac{1}{2}(p_0^2+q_0^2) + \frac{1}{2}\sigma^2 t_n + C_{\alpha_1}(h)\sigma^2 t_n,
\end{equation}
where
\begin{equation}\label{e:Cah}
	C_{\alpha_1}(h) = -\frac{\left(\sqrt{6} \alpha_1+2 \sqrt{\frac{1}{\alpha_1}-1}-\sqrt{6}\right) h^2 \left((\alpha_1-1) \alpha_1 h^2+4\right)}{6 \sqrt{\frac{1}{\alpha_1}-1} \big((\alpha_1-1)^2 h^2+4\big) \left(\alpha_1^2 h^2+4\right)}.
\end{equation}

\begin{figure}[htbp]
	\centering
	\includegraphics[width=1\textwidth]{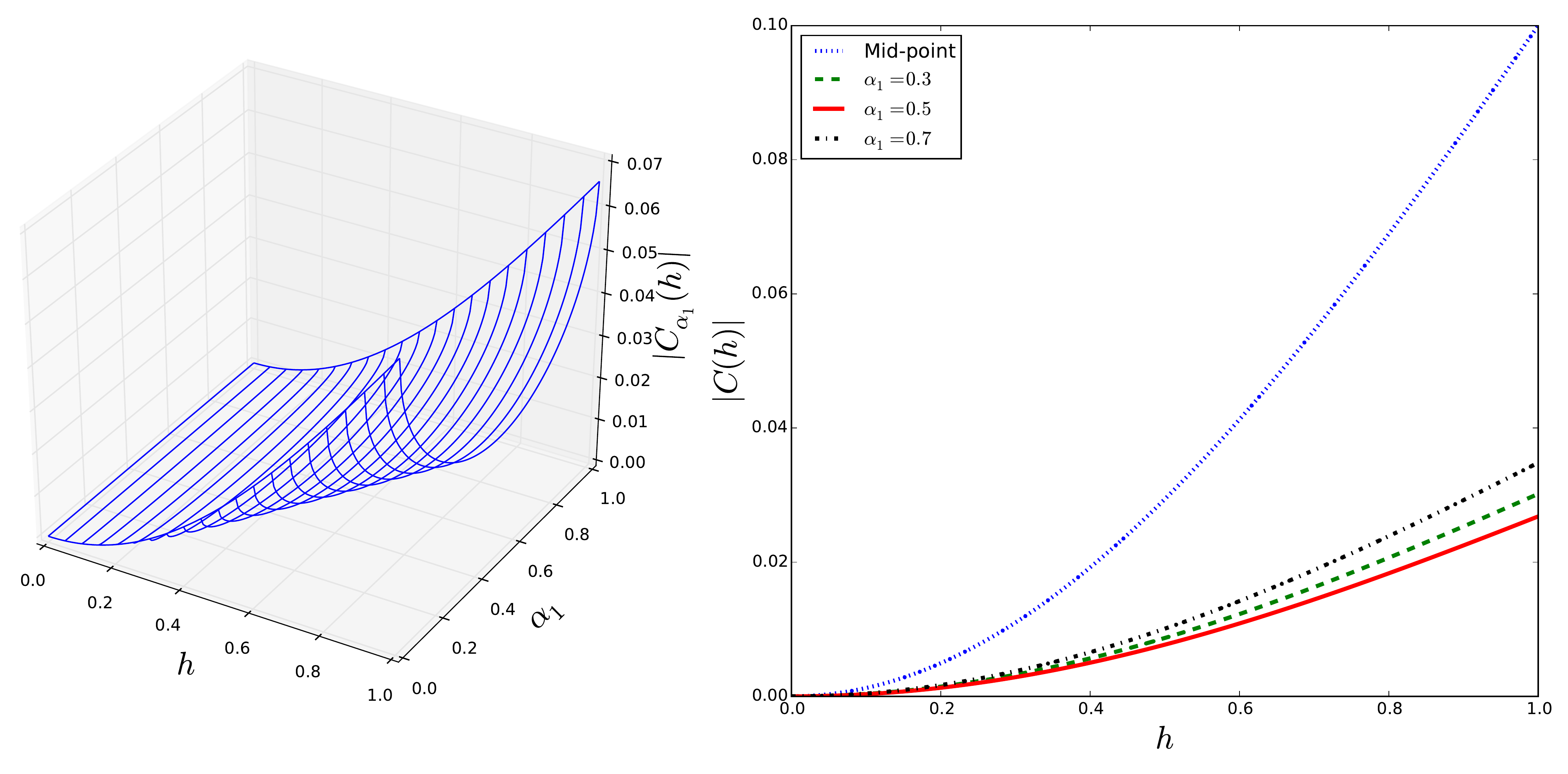}
	\caption{Left: surface of the $|C_{\alpha_1}(h)|$ \eqref{e:Cah} with  $h$ and $\alpha_1$. Right: the increase of $|C(h)|$ of mid-point method and SSRK-$\alpha_1$ with $\alpha_1$ = 0.3, 0.5, 0.7, respectively.}
	\label{fig:alpha1_symplectic}
\end{figure}

The left-hand side of Figure \ref{fig:alpha1_symplectic} plots the surface of $|C_{\alpha_1}(h)|$ with $h, \alpha_1 \in (0,1)$, from which we observe that there is only one $\alpha_1$ to minimize the $C_{\alpha_1}(h)$ for every $h$ in this domain. The optimal parameter $\alpha_{opt}$ is
\begin{equation}
	\argmin_{0<\alpha_1<1} |C_{\alpha_1}(h)| = \frac{1}{2}, \quad\text{for every} \, h\in (0,1),
\end{equation}
by common calculation. Therefore, when considering the preservation of the linear growth property \eqref{e:linear_growth}, the SSRK-0.5 scheme (i.e., $\alpha_1$=0.5)
is the optimal one of the SSRK-$\alpha_1$ family. For example, values of
$|C(h)|$ of four schemes with respect to $h\in (0,1)$ are shown in the right-hand side of Figure \ref{fig:alpha1_symplectic} where the Mid denotes the mid-point method (symplectic and of mean-square 1.0\cite{Hong2006mid}). From these curves, the one of $\alpha_1$ grows slowest, which  indicates that the SSRK-0.5 scheme shows better ability in preserving the linear growth property \eqref{e:linear_growth}.

Moreover, if we repeat the same analysis for the two-parameter SSRK-$\alpha_1$-$b_1$ methods in Table \ref{tab:SSRKab} and fix $\alpha_1=\frac{1}{2}$, then we obtain 
\begin{equation}
	C_{b_1}(h) = \frac{(3\sqrt{2-3b_1^2}-3b_1-4)(16-h^2)h^2}{3(16+h^2)}, \quad -\sqrt{\frac{2}{3}}<b_1<\sqrt{\frac{2}{3}}.
\end{equation} 
So when $b_1 = \frac{-2\pm\sqrt{6}}{6}$ we have $C_{b_1}(h)$=0, which means the mean-square order 2.0 scheme SSRK-0.5-$\frac{-2\pm\sqrt{6}}{6}$ preserves the linear growth property exactly for every $h$.

\subsection{Non-linear stochastic oscillator}
Next we consider a non-linear oscillator of the form:
\begin{equation}\label{e:nonl_oscillator}
\left\{
\begin{aligned}
dP &= (Q-Q^3)dt + \sigma_1 dW_1(t) + \sigma_2 dW_2(t), & P(0) = p_0, \\
dQ &= P dt, & Q(0) = q_0,
\end{aligned}\right.
\end{equation}
which is referred to as the double well problem \cite{Burrage2012}. It is also a  second-order Hamiltonian system with additive noise \eqref{e:order2_hamilton}, where $M = 1$ and $U(q) = -\frac{1}{2}q^2 + \frac{1}{4}q^4$. Then, $H_0(p,q) = \frac{1}{2}(p^2 - q^2) + \frac{1}{4}q^4$ and  $H_i(p,q) = -\sigma_i q$ for $i=1,\, 2$. It turns out that
\begin{equation}\label{e:nonlinear_growth}
	\E\Big(H_0\big(P(t),Q(t)\big)\Big) = H_0(p_0,q_0) + \frac{1}{2}(\sigma_1^2+\sigma_2^2) t,
\end{equation}
which shows a linear growth property as well.

\begin{table}[htbp]
	\centering
	\caption{Mean-square errors for \eqref{e:nonl_oscillator} with different schemes} 
	\label{tab:nonlinear_osc}
	\begin{tabular*}{1\textwidth}{@{\extracolsep{\fill}}@{}lcccccc@{}}
		\toprule
		$h$   & $2^{-2}$ & $2^{-3}$ & $2^{-4}$ & $2^{-5}$ & $2^{-6}$ & order \\ \midrule
		Euler & 6.35E-01 & 3.32E-01 & 1.53E-01 & 6.83E-02 & 2.90E-02 & 1.12\\
		SRK-0.5  & 3.04E-01& 6.55E-02& 1.53E-02& 3.64E-03& 8.62E-04 & 2.11\\
		SSRK-0.5  & 2.73E-02& 7.37E-03& 1.93E-03& 4.67E-04& 1.15E-04 & 1.99  \\ \bottomrule
	\end{tabular*}
\end{table}

As before, we check the mean-square convergence of SSRK methods for system \eqref{e:nonl_oscillator}. In this test, we set $\sigma_1=\sigma_2=1$,  $(p_0, q_0)=(1, 0)$, and $T=1$. The results of mean-square errors are listed in Table \ref{tab:nonlinear_osc}, and shown in Figure \ref{f:nonlinearosc_order} directly. From them, we observe that both SRK-0.5 and SSRK-0.5 are mean-square 2.0, which accords with the analysis of Theorem \ref{thm:2order_add}.

\begin{figure}[htbp]
	\centering
	\includegraphics[width=1\textwidth]{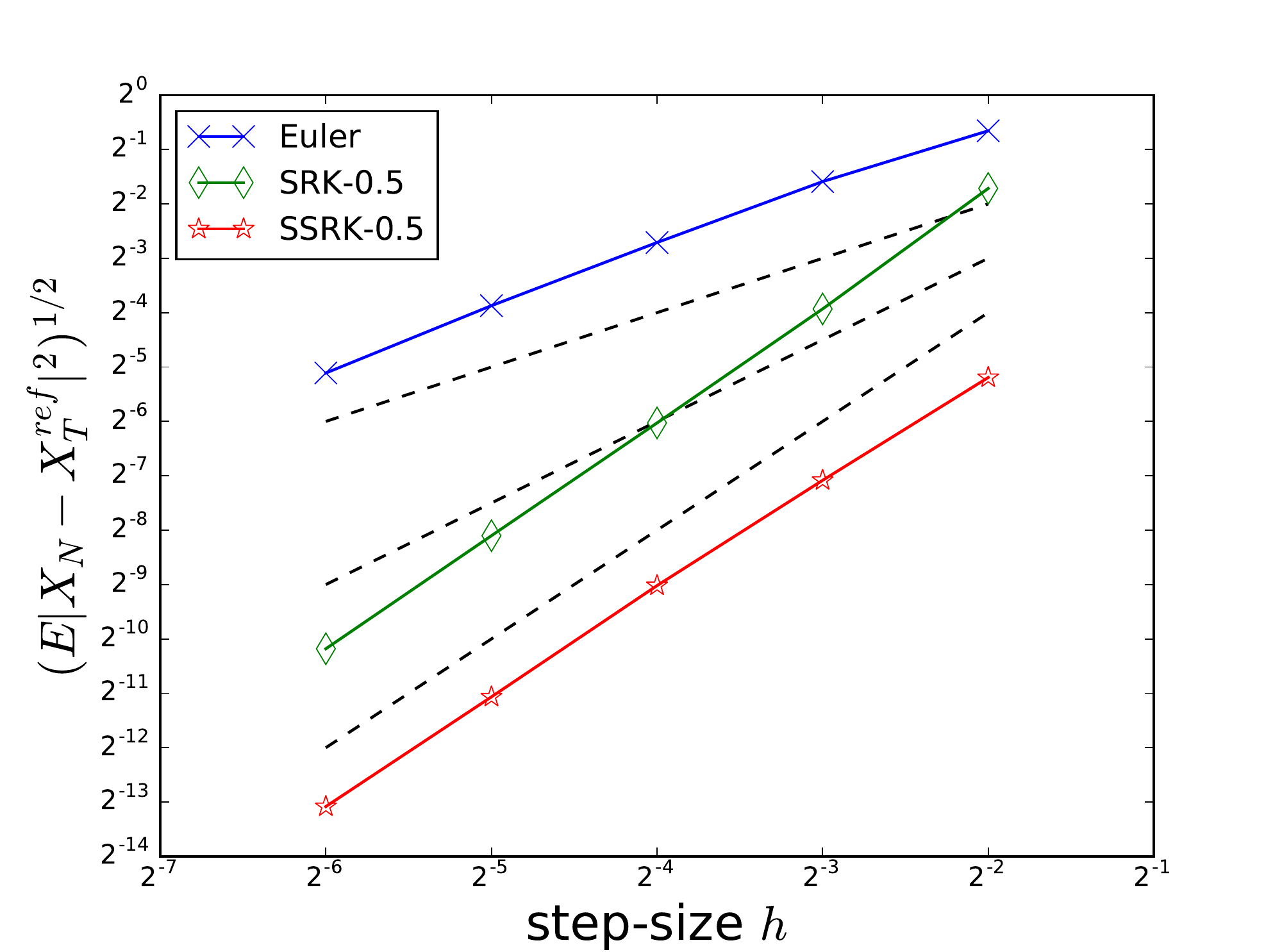}
	\caption{Mean-square errors at $T=1$ of  three methods for non-linear oscillator \eqref{e:nonl_oscillator} with $\sigma = 0.5$, $(p_0 , q_0)=(\sqrt{2}, \sqrt{2})$. The dashed reference lines have slopes 1, 1.5 and 2. respectively}\label{f:nonlinearosc_order}
\end{figure}

Note that the linear growth property \eqref{e:nonlinear_growth} also holds for this non-linear system, so we check it for the SRK methods numerically. Following \cite{Burrage2012}, we set $\sigma_1=0.5$, $\sigma_2=0$, initial value $(p_0, q_0)=(\sqrt{2},\sqrt{2})$, step-size $h=0.1$, and $T=40$. The left part of Figure \ref{fig:energy_Non} depicts average of $H_0$ along the numerical solutions by SRK-0.5 and SSRK-0.5 schemes over 50000 trajectories, and we do not plot the result of Euler scheme because it has exponential growth. We can observe that the SSRK-0.5 scheme preserves this linear property quite well which coincides with the reference line. Besides, the right part of Figure \ref{fig:energy_Non} demonstrates the average solution versus time of the SSRK-0.5 scheme, which is similar to the result in \cite{Burrage2012}.

\begin{figure}[htbp]
	\centering
	\includegraphics[width=1\textwidth]{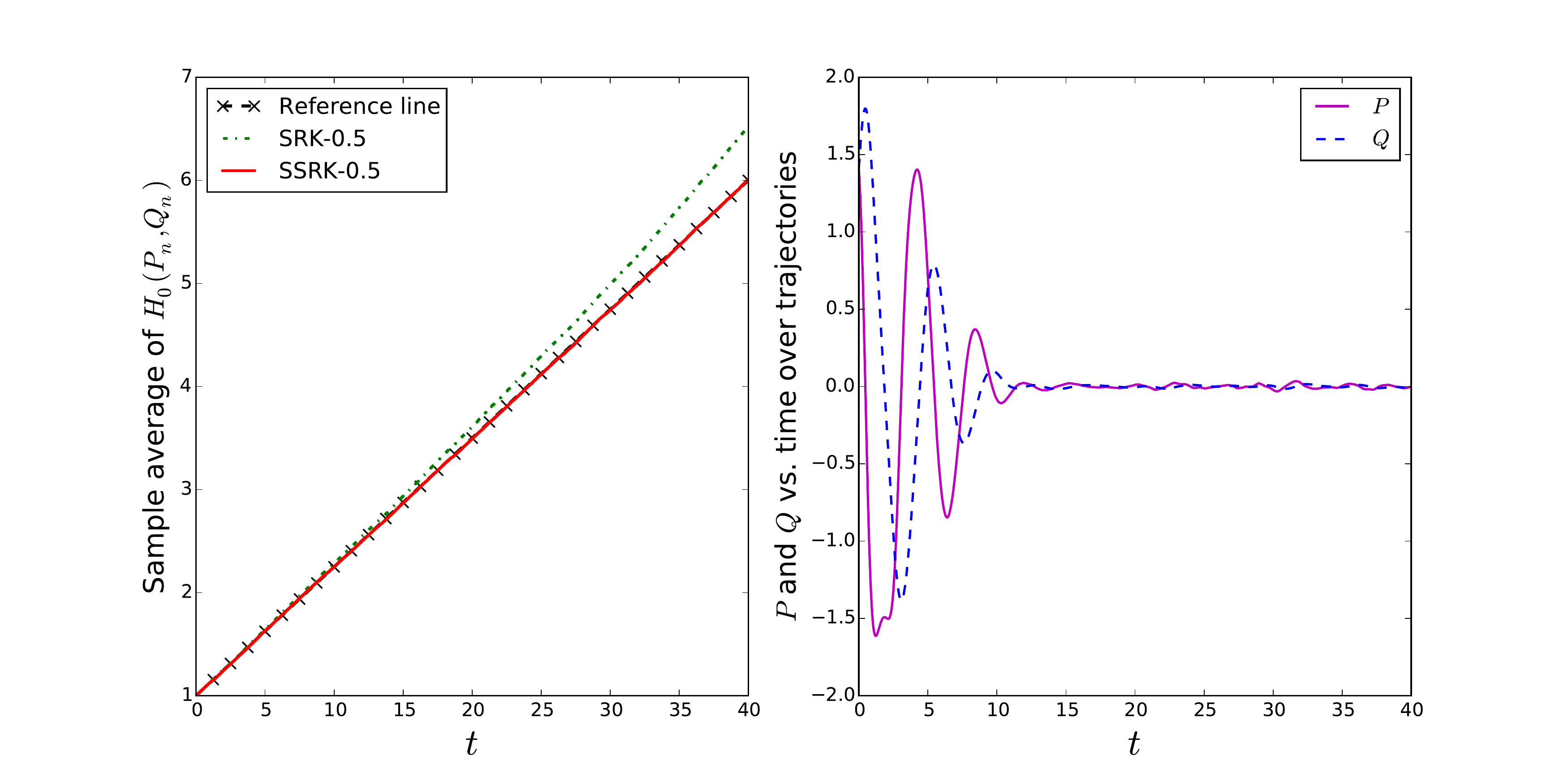}
	\caption{Non-linear oscillator \eqref{e:nonl_oscillator}. Left: growth rate for $E(H_0)$ \eqref{e:nonlinear_growth} along numerical solutions with time interval $t\in[0,40]$. Right: average of numerical solutions over 50000 trajectories.}
	\label{fig:energy_Non}
\end{figure}

\section{Conclusions}
In the present paper, we investigate  SRK methods aiming at constructing stochastic symplectic methods for stochastic Hamiltonian systems with additive noise. Using colored rooted tree theory, the conditions of mean-square order 1.5 in general case and 2.0 for the second-order Hamiltonian systems are obtained, under which we propose two classes of 2-stage SSRK methods (SSRK-$\alpha_1$ and SSRK-$\alpha_1$-$b_1$) combined with the common symplectic conditions. Numerical experiments are finally performed to the linear and non-linear Hamiltonian systems with additive noise to verify the mean-square order theory. Moreover, the linear growth property of this kind of system is especially taken into consideration. We find that the proposed SSRK methods have very good ability in preserving this property due to symplecticity. Especially for a linear oscillator with additive noise, choosing proper coefficients in the SSRK-$\alpha_1$-$b_1$ methods, we even obtain schemes exactly preserving the linear growth property, which are of mean-square order 2.0 as well.

For separable Hamiltonian systems with additive noise, we can also construct high order explicit and fully derivative-free symplectic schemes by using partitioned Runge--Kutta (PRK) methods \cite{Hairer2006} in the stochastic case, which deserves further investigation.

\bibliographystyle{elsarticle-num}
\bibliography{srkAdd}   % name your BibTeX data base

\end{document}